\documentclass[12pt,fleqn]{article}
\usepackage[margin=1.0in]{geometry}
\usepackage[utf8]{inputenc}
\usepackage{amsfonts}
\usepackage[tbtags]{amsmath}
\usepackage{amsthm}
\usepackage{amssymb}

\newtheorem{theorem}{Theorem}
\newtheorem{lemma}{Lemma}

\newcommand{\be}{\begin{equation}}
\newcommand{\ee}{\end{equation}}

\newcommand{\wt}{\widetilde}
\newcommand{\wh}{\widehat}

\newcommand{\De}{\Delta}
\newcommand{\al}{\alpha}

\newcommand{\si}{\sigma}

\renewcommand{\th}{\theta}

\newcommand{\Q}{\mathbb{Q}}
\newcommand{\Z}{\mathbb{Z}}

\newcommand{\bi}{\bibitem}

\title{On the spectrum of critical almost Mathieu operators in the rational case}
\author{S. Jitomirskaya, L. Konstantinov, I. Krasovsky}
\date{}

\begin{document}
\maketitle
\begin{abstract} 
We derive a new Chambers-type formula and prove sharper upper bounds on the
measure of the spectrum of critical almost Mathieu operators with
rational frequencies.
\end{abstract}

\section{Introduction}
The Harper operator, a.k.a. the discrete magnetic Laplacian\footnote{The
  name ``discrete magnetic Laplacian'' was first introduced by M. Shubin in \cite{shu}.}, is
a tight-binding model of an electron confined
to a 2D square lattice in a
uniform magnetic field orthogonal to the lattice plane and with flux $2\pi\alpha$ through an elementary cell. 
It acts on
$\ell^2(\Z^2)$ and
is usually given in the Landau gauge representation
\begin{equation} \label{alpha}
(H(\alpha)\psi)_{m,n}=\psi_{m,n-1}+\psi_{m,n+1}+e^{-i2\pi\alpha
  n}\psi_{m-1,n}+e^{i2\pi\alpha n}\psi_{m+1,n},
\end{equation}
first considered by Peierls \cite{Peierls}, who noticed that it makes the Hamiltonian separable and
turns it into the direct integral in $\theta$ of operators
on $\ell^2(\mathbb{Z})$ given by:
\begin{equation}
(H_{\alpha, \theta} \varphi)(n) = \varphi(n - 1) + \varphi(n + 1) + 2 \cos 2 \pi (\alpha n + \theta) \varphi(n), \qquad \alpha, \theta \in [0, 1).
\end{equation}

In physics literature, it also appears under the names Harper's or  the Azbel-Hofstadter
model, with both names  used also for the discrete magnetic
Laplacian $H(\alpha)$.  In mathematics, it is universally called the critical almost Mathieu
operator.\footnote{This name was originally introduced by Barry Simon
  \cite{flu}.} In addition to importance in physics, this model is of 
special interest, being at the boundary of two reasonably well
understood regimes: (almost) localization and (almost) reducibility,
and not being amenable to methods of either side. Recently, there has been some
progress in the study of the fine structure of its spectrum \cite{hlqz,j,jz,Kcentral,ls}. 
 
Denote the spectrum of an operator $H, $ as a set, by $\si(H)$. An important object is the union of  $\si(H_{\alpha, \theta})$ over
$\theta,$ which coincides with the spectrum of $H(\alpha)$. We denote it $S(\alpha) :=\sigma(H(\alpha))= \cup_{\theta \in [0, 1)} \sigma(H_{\alpha, \theta})$. 
Note that by the general theory of ergodic operators, if $\al$ is irrational,
 $\si(H_{\alpha, \theta})$ is independent of $\th$. We denote the Lebesgue measure of a set $A$ by $|A|$.

For irrational $\al$, the Lebesgue measure $|S(\al)|=0$, and $S(\al)$
is a set of Hausdorff dimension no greater than $1/2$
\cite{Last94,AK,JK19}. The proof of the Hausdorff dimension result in \cite{JK19}
(which was a conjecture of D. J. Thouless)  is
based on upper bounds of the measure of the spectrum for $\alpha\in\Q$
 and a strong continuity. For rational $\al=\frac{p_0}{q_0}$, where $p_0$, $q_0$ are coprime positive integers, Last obtained the bounds \cite[Lemma 1]{Last94}:
\begin{equation} \label{eq17}
\frac{2(\sqrt{5}+1)}{q_0}<\left| S\left(\frac{p_0}{q_0}\right)\right| < \frac{8e}{q_0},
\end{equation}
where $e=\exp(1)=2.71\dots$. While the upper bound in (\ref{eq17}) was sufficient for the
argument of \cite{JK19}, the measure of the spectrum is subject to another
conjecture of Thouless \cite{T83,Tcmp}: that
 in the limit $p_n/q_n\to \al$, we have
$q_n|S(p_n/q_n)|\to c$, where $c=32C_c/\pi=9.32\dots$, 
$C_c=\sum_{k=0}^{\infty}(-1)^k(2k+1)^{-2}$ being  the Catalan constant. 
Thouless provided a partly heuristic argument in the case $p_n=1$, $q_n\to\infty$.
A rigorous proof  for
$\alpha=0$ and $p_n=1$ or $p_n=2$, $q_n$ odd, was given in \cite{hk}.

The purpose of this note is to present a sharper upper bound, for all $\alpha\in\Q$:
\begin{theorem} \label{Thm1}
For all positive coprime integers $p_0$ and $q_0$, 
\[
\left| S\left(\frac{p_0}{q_0}\right)\right| \leq \frac{4\pi}{q_0}.
\]
\end{theorem}
Thus, the upper bound is reduced from $8e=21.74\dots$ to $4\pi= 12.56\dots$.
The way we prove Theorem 1 is very different from that of \cite{Last94}; we use the chiral gauge representation 
\cite{JK19} and Lidskii's inequalities. The chiral gauge
representation of the almost Mathieu operator also leads to 
a new type of Chambers' relation (equations \eqref{Chambersodd}, \eqref{Chamberseven} below).

\section{Proof of Theorem \ref{Thm1}}
Consider the following operator on $\ell^2(\mathbb{Z})$:
\begin{multline}\label{cg}
(\widetilde{H}_{\alpha, \theta} \varphi)(n) = 2 \sin 2 \pi (\alpha (n-1) + \theta) \varphi(n - 1) + 2 \sin 2 \pi (\alpha n + \theta) \varphi(n + 1), \qquad
 \alpha, \theta \in [0, 1),
\end{multline}
and define $\widetilde{S}(\alpha) := \cup_{\theta \in [0, 1)} \sigma(\widetilde{H}_{\alpha, \theta})$.
It was shown in \cite[Theorem 3.1]{JK19} that the operators $M_{2 \alpha} := \oplus_{\theta \in [0, 1)} H_{2\alpha, \theta}$ and $\widetilde{M}_{\alpha} := \oplus_{\theta \in [0, 1)} \widetilde{H}_{\alpha, \theta}$ are unitarily equivalent, so 
that $S(\al) = \widetilde{S}(\al/2) $. (Note that $\sigma(H_{2\alpha, \theta}) \neq \sigma(\widetilde{H}_{\alpha, \theta}),$ in general.) See also related partly non-rigorous considerations in
\cite{mz,koh,wznp,kprb,kdiscr}, and an application of the rational case in \cite{Kcentral}.
Operator \eqref{cg} corresponds to the chiral gauge representation of the Harper operator.

%Further, let $\Delta^H_{\frac{p_0}{q_0}}$ denote the discriminant of $H_{\frac{p_0}{q_0}, \theta}$ (cf. \cite{Last94}). 

From now on, we always consider the case of rational $\alpha$. Furthemore, the analysis below for $q_0=1$, $q_0=2$ becomes especially elementary, and gives $|S(1)|=8$, $|S(1/2)|=4\sqrt{2}$, so that Theorem \ref{Thm1} obviously holds in these cases. From now on, we assume $q_0\ge 3$.

If $p_0$ is even, define $p:=\frac{p_0}{2}$ and $q:=q_0$ (note that
$q$ is necessarily odd in this case). %Then $p$ and $q$ are
                                %coprime. We have  $| S(p_0/q_0) | = |
                                %\widetilde{S}(p/q) |$. 
 This corresponds to case I below.
If $p_0$ is odd, define $p:=p_0$ and $q:=2q_0$. 
This corresponds to case II below.
We note that in either case $p$ and $q$ are coprime and $S(p_0/q_0)  =  \widetilde{S}(p/q) $. 

%We shall use these relations to pass from those mathematical objects corresponding to $\widetilde{H}_{\frac{p}{q}, \theta}$, to those corresponding to $H_{\frac{p_0}{q_0}, \theta}$, and vice versa.

Let $b(x) := 2 \sin (2 \pi x)$, and further identify $b_n(\theta) := b((p/q) n + \theta)$.
For the operator $\widetilde{H}_{\frac{p}{q}, \theta}$, Floquet theory
states %(see, e.g., \cite{Last92}) 
that $E \in \sigma(\widetilde{H}_{\frac{p}{q}, \theta})$ if and only if the equation $(\widetilde{H}_{\frac{p}{q}, \theta} \varphi)(n) = E \varphi(n)$ has a solution $\{ \varphi(n) \}_{n \in \mathbb{Z}}$ satisfying $\varphi(n+q) = e^{i k q} \varphi(n)$ for all $n$, and for some real $k$. Therefore, for a fixed $k$, there exist $q$ values of $E$ satisfying the eigenvalue equation
\begin{equation}
B_{\theta, k,\ell}
\begin{pmatrix}
\varphi(\ell) \\
\vdots \\
\varphi(\ell+q-1)
\end{pmatrix} = E
\begin{pmatrix}
\varphi(\ell) \\
\vdots \\
\varphi(\ell+q-1)
\end{pmatrix}
\end{equation}
for any $\ell$, where
\be \label{eq16}
B_{\theta, k,\ell} :=
\begin{pmatrix}
0 & b_{\ell} & 0 & 0 & \cdots & 0 & 0 & e^{-i k q} b_{\ell+q-1} \\
b_{\ell} & 0 & b_{\ell+1} & 0 & \cdots & 0 & 0 & 0 \\
0 & b_{\ell+1} & 0 & b_{\ell+2} & \cdots & 0 & 0 & 0 \\
\vdots & \vdots & \vdots & \vdots & \ddots & \vdots & \vdots & \vdots \\
0 & 0 & 0 & 0 & \cdots & b_{\ell+q-3} & 0 & b_{\ell+q-2} \\
e^{i k q} b_{\ell+q-1} & 0 & 0 & 0 & \cdots & 0 & b_{\ell+q-2} & 0
\end{pmatrix}.
\ee
Thus, the eigenvalues of $B_{\theta, k,\ell}$ are independent of
$\ell$. 

\subsection{Chambers-type formula}

% {\bf Remark} {\it
% Equations \eqref{Chambersodd} and \eqref{Chamberseven} are variants of
% the well-known Chambers' relation for $H_{\frac{p_0}{q_0}, \theta}$
% \cite{Chambers65} (see \cite{BellissardSimon82} for a proof):

The celebrated Chambers' formula presents the dependence of the
determinant of the almost Mathieu operator with $\alpha=p_0/q_0$
restricted to the period $q_0$ with Floquet boundary conditions, on
the phase $\theta$ and quasimomentum $k$. In the critical case it is
given by (see, e.g., \cite{Last94})
\be \label{ChambersFormula}
\det(A_{\theta, k, \ell} - E) = \Delta(E) -2 (-1)^{q_0}(\cos(2 \pi q_0 \theta) + \cos(k q_0)),
\ee
where 
\be
A_{\theta, k,\ell} :=
\begin{pmatrix}
a_{\ell} & 1 & 0 & 0 & \cdots & 0 & 0 & e^{-i k q} \\
1 & a_{\ell+1} & 1 & 0 & \cdots & 0 & 0 & 0 \\
0 & 1 & a_{\ell+2} & 1 & \cdots & 0 & 0 & 0 \\
\vdots & \vdots & \vdots & \vdots & \ddots & \vdots & \vdots & \vdots \\
0 & 0 & 0 & 0 & \cdots & 1 & a_{\ell+q-2} & 1 \\
e^{i k q} & 0 & 0 & 0 & \cdots & 0 & 1 & a_{\ell+q-1}
\end{pmatrix}, \,\,\, \ell \in \mathbb{Z},
\ee
\be
a(x) := 2 \cos(2 \pi x), \qquad a_n(\theta) := a((p_0/q_0) n + \theta),
\ee
and $\Delta$, the discriminant\footnote{In \cite{Last94}, the discriminant differs from $\Delta(E)$ by the factor $(-1)^{q_0}$.}, is independent of $\theta$ and $k$. An
immediate corollary of this formula is  that $S\left(\frac{p_0}{q_0}
\right) = \Delta^{-1}([-4, 4])$, e.g., \cite{Last94}. 

Here we obtain a formula of this type for 
$\det(B_{\theta, k,\ell} - E).$ 
Indeed, as usual, separating the terms containing $k$ in the determinant, we
obtain, 
for the characteristic polynomial $D_{\theta, k}(E) := \det(B_{\theta, k,\ell} - E):$
\begin{equation} \label{eq1}
D_{\theta, k}(E) 
= D^{(0)}_{\theta}(E)  - (-1)^q b_0 \cdots b_{q-1} \cdot 2 \cos(kq),
\end{equation}
where  $D^{(0)}_{\theta}(E) $ is  independent of $k$ and equal therefore to $D_{\theta, k = \frac{\pi}{2 q}}(E)$.

For the product of $b_j$'s we have:
\begin{lemma}\label{prod}
\begin{equation}\label{bbb}
\begin{aligned}
b_0 \cdots b_{q-1} &= \prod_{j = 0}^{q-1} 2 \sin 2\pi\left(\frac{p}{q} j + \theta\right)\\
&=4 \sin(\pi q \theta) \sin \pi q (\theta + 1/2) = 2( \cos(\pi q/2) - \cos \pi q (2\theta + 1/2) ).
\end{aligned}
\end{equation}
\end{lemma}

\begin{proof}
To evaluate the product of $b_j$'s, we expand sine in terms of exponentials and use the formula
$1 - z^{-q} = \prod_{j = 0}^{q-1} (1 - z^{-1} e^{2 \pi i \frac{p}{q} j})$.
 An alternative derivation can go along the lines of the proof
  of Lemma 9.6 in \cite{martini}. 
\end{proof}

Substituting (\ref{bbb}) into (\ref{eq1}), we have
\begin{equation} \label{eq2}
D_{\theta, k}(E) = D^{(0)}_{\theta}(E) - 8(-1)^q \sin(\pi q \theta) \sin \pi q (\theta + 1/2) \cos(kq).
\end{equation}

We can further obtain the dependence of $D^{(0)}_{\theta}(E)$ on $\theta$:

\begin{lemma} \label{lemma1}
$$
D^{(0)}_{\theta}(E) = \wt\Delta(E) + \left\{
\begin{array}{ll}
0, & q \mbox{ odd} \\
4(\cos(2 \pi q \theta)-1), & q \mbox{ even},
\end{array} 
\right.
$$
where the discriminant $\wt\Delta(E):=D^{(0)}_{\theta=0}(E)$ is independent of $\theta$.
\end{lemma}

\begin{proof}
Since $D_{\theta, k}(E)$ is independent of $\ell$, it is
$1/q$ periodic in $\theta$, i.e.,
$D_{\theta,k}(E)=D_{\theta+1/q,k}(E)$, and by \eqref{eq1} so is $D^{(0)}_{\theta}(E)$.
Therefore, since, clearly,  $D^{(0)}_{\theta}(E)=\sum_{n=-q}^q c_n (E)
e^{2\pi i \theta n} $, the terms $c_k$ other than $k=mq$ vanish, and $D^{(0)}_{\theta}(E)$ has the following Fourier expansion:
$$D^{(0)}_{\theta}(E) = c_0(E)+c_q e^{2\pi i q \theta}+c_{-q}e^{-2\pi i q \theta}.$$

It is easily seen that the $c_q$ and $c_{-q}$ can be obtained from the expansion of
the determinant and that, moreover,
they do not depend on $E. $ Expanding $D^{(0)}_{\theta}(E)$ with $E=0$ in rows and columns (cf. \cite{Kcentral}),
we obtain
\begin{equation}
D^{(0)}_{\theta}(0) = D_{\theta, k=\frac{\pi}{2q}}(0)=
\left\{
\begin{array}{ll}
0, & q \mbox{ odd} \\
(-1)^{q/2} (b_0^2 b_2^2 \cdots b_{q-2}^2 + b_1^2 b_3^2 \cdots b_{q-1}^2), & q \mbox{ even}.
\end{array} 
\right.
\end{equation}

This gives $c_q=c_{-q}=0$ for $q$ odd, and $c_q=\prod_{j=0}^{\frac{q-2}2} e^{8\pi i \frac{p}{q}j}+\prod_{j=0}^{\frac{q-2}2} e^{4\pi i \frac{p}{q}(2j+1)}=2=c_{-q},$ for $q$ even. It remains to denote 
$\wt\De(E)=c_0(E)$ for $q$ odd, and $\wt\De(E)=c_0(E)+4$ for $q$ even, and the proof is complete.
\end{proof}

We therefore have, by   (\ref{eq2}) and Lemma \ref{lemma1}:

\begin{lemma}[Chambers-type formula]
\begin{equation}\label{Chambersodd}
D_{\theta, k}(E)= \widetilde\Delta(E) + 4(-1)^{(q-1)/2}  \sin(2 \pi q \theta) \cos(k q),\qquad q\,\mbox{ odd}.
\end{equation}
\begin{equation} \label{Chamberseven}
D_{\theta, k}(E) = \wt\Delta(E) - 4(1 - \cos(2 \pi q \theta))(1 + (-1)^{q/2}\cos(k q)),\qquad q\,\mbox{ even}.
\end{equation}
\end{lemma}

Note that $\widetilde\Delta(E)$ is a polynomial of degree $q$ independent of $k\in\mathbb R$ and $\theta\in[0,1)$.
By Floquet theory, the spectrum $\sigma(\widetilde{H}_{\frac{p}{q}, \theta})$ is the union of the eigenvalues of
$B_{\th,k, \ell}$ over $k$, a collection of $q$ intervals. 

We make the following observations.

\bigskip

\underline{Case I: q is odd.}

By \eqref{Chambersodd}, 
$D_{\theta, k}(E)\equiv \det(B_{\theta,k,\ell}-E) = 0$ if and only if
$\wt\Delta(E) = 4 (-1)^{(q+1)/2} \sin(2 \pi q \theta) \cos(k q)$. 
Thus, $\sigma(\widetilde{H}_{\frac{p}{q}, \theta})$ is the preimage of $[-4|\sin(2 \pi q \theta)|, 4|\sin(2 \pi q \theta)|]$
under the mapping $\wt\Delta(E)$.  If $\th=m/(2q)$, $m\in\mathbb Z$, $\sigma(\widetilde{H}_{\frac{p}{q}, \frac{m}{2q}})$
is a collection of $q$ points where $\wt\Delta(E)=0$. (In this case, $b_0(m/(2q))=0$,
so that $\wt H$ splits into the direct sum of an infinite number of copies of a $q$-dimensional matrix.) 
We note that the spectra $\sigma(\widetilde{H}_{\frac{p}{q}, \theta})$ for
different $\theta$ are nested in one another as $\theta$ grows from $0$ to $1/(4q)$; in particular,
for each $\theta \in [0, 1)$,
\begin{equation} \label{eq5}
\sigma(\widetilde{H}_{\frac{p}{q}, \theta})=\wt\Delta^{-1}([-4|\sin(2 \pi q \theta)|, 4|\sin(2 \pi q \theta)|]) 
 \subseteq \sigma(\widetilde{H}_{\frac{p}{q}, \theta = \frac{1}{4 q}})=\wt\Delta^{-1}([-4, 4]).
\end{equation}
This implies that all the maxima of $\wt\Delta(E)$ are no less than 4, and all the minima are no greater than $-4$.
Moreover, taking the union over all $\theta \in [0, 1)$ gives:
\begin{equation}
\widetilde{S}\left(\frac{p}{q} \right) =\sigma(\widetilde{H}_{\frac{p}{q}, \theta = \frac{1}{4 q}})=\wt\Delta^{-1}([-4, 4]).
\end{equation}
Clearly, it is sufficient to consider only $\th\in [0,1/(4q)]$.

\bigskip

\underline{Case II: q is even.}
This case is similar to case I, so we omit some details for brevity.
By \eqref{Chamberseven}, 
$D_{\theta, k}(E) = 0$ if and only if
$\wt\Delta(E) = 4(1 - \cos(2 \pi q \theta))(1 + (-1)^{q/2}\cos(k q))$.
Considering the cases $k=0, \frac{\pi}{q}$, it is easy to see that $\sigma(\widetilde{H}_{\frac{p}{q}, \theta})$ is the 
preimage of $[0, 8 - 8 \cos(2 \pi q \theta)]$
under the mapping $\wt\Delta(E)$.  If $\th=m/q$, $m\in\mathbb Z$, $\sigma(\widetilde{H}_{\frac{p}{q}, \frac{m}{q}})$
is a collection of $q$ points where $\wt\Delta(E)=0$.
We note that the spectra $\sigma(\widetilde{H}_{\frac{p}{q}, \theta})$ for
different $\theta$ are nested in one another as $\theta$ grows from $0$ to $1/(2q)$; in particular,
for each $\theta \in [0, 1)$,
\begin{equation} \label{eq14}
\sigma(\widetilde{H}_{\frac{p}{q}, \theta})=\wt\Delta^{-1}([0, 8 - 8 \cos(2 \pi q \theta)]) 
 \subseteq \sigma(\widetilde{H}_{\frac{p}{q}, \theta = \frac{1}{2 q}})=\wt\Delta^{-1}([0, 16]).
\end{equation}
This implies that all the maxima of $\wt\Delta(E)$ are no less than 16, and all the minima are no greater than 0.
Moreover, taking the union over all $\theta \in [0, 1)$ gives:
\begin{equation}
\widetilde{S}\left(\frac{p}{q} \right) =\sigma(\widetilde{H}_{\frac{p}{q}, \theta = \frac{1}{2 q}})=\wt\Delta^{-1}([0, 16]).
\end{equation}
Clearly, it is sufficient to consider only $\th\in [0,1/(2q)]$. 

In this case of even $q$ we can say more about the form of $\wt\Delta(E)$. Note that $b_0(0)=b_{q/2}(0)=0$
and $b_k(0)=b_{-k}(0)$. Recall that by Floquet theory, $D_{\theta,k}(E)=\det(B_{\theta,k,\ell}-E)$ is independent
of the choice of $\ell$. For convenience, choose $\ell=-q/2+1$. It is easily seen that $B_{\theta=0,k,\ell=-q/2+1}$ 
decomposes into a direct sum,
and moreover $\wt\Delta(E)=D_{\theta=0,k}(E)=(-1)^{q/2}P_{q/2}(-E)P_{q/2}(E)$, where $P_{q/2}(E)$ is a polynomial 
of degree $q/2$, odd if $q/2$ is odd, and even if $q/2$ is even (as it is a characteristic polynomial of a tridiagonal matrix with zero main diagonal). Thus $\wt\Delta(E)=P_{q/2}(E)^2$ is a square. 

\bigskip

The discriminants $\wt\Delta(E)\equiv \wt\Delta_{p/q}(E)$ and $\Delta(E)\equiv \Delta_{p_0/q_0}(E)$ are related in the following way:
\begin{lemma}
For $q$ odd,
\begin{equation}
\wt\Delta_{p/q}(E)=\Delta_{p_0/q_0}(E),\qquad p_0=2p,\quad q_0=q.
\end{equation}
For $q$ even,
\begin{equation}
\wt\Delta_{p/q}(E)=\Delta_{p_0/q_0}^2(E),\qquad p_0=p,\quad q_0=q/2.
\end{equation}
\end{lemma}

\begin{proof}

\underline{Case I: q is odd.}
Here, by our definitions at the start of the section, $p_0 = 2 p$ and $q_0 = q$. $\wt\Delta_{p/q}(E)$ and 
$\Delta_{p_0/q_0}(E)$ are polynomials in $E$ of degree $q$ with the same coefficient $-1$ of $E^q$.
Since $\wt\Delta(E) = \Delta(E) = \pm 4$ at the $2q \geq q+1$ distinct edges of the bands (cf. \cite[3.3]{Choietal90}), these polynomials coincide: $\wt\Delta(E) = \Delta(E)$ for each $E$.

\underline{Case II: q is even.}
Here, $p_0 = p$ and $q_0 = q/2$. $\wt S\left( \frac{p}{q} \right) =
S\left( \frac{p_0}{q_0} \right)$ is the preimage of $[0, 16]$ under
$\wt\Delta_{p/q}$ and of $[-4, 4]$ under $\Delta_{p_0/q_0}$, hence
also of $[0, 16]$ under $\Delta_{p_0/q_0}^2$. On the other hand, we have seen above that $\wt\Delta(E)=P_{q/2}^2(E)$
for some polynomial $P_{q/2}(E)$ of degree $q/2=q_0$.
Thus, $P_{q/2}^2(E)$
and $\Delta^2(E)$ coincide at the $2q_0 \geq q_0+1$ (for $q_0$ odd) and 
$2q_0-1 \geq q_0+1$ (for $q_0$ even) distinct edges of
the bands (cf. \cite[3.3]{Choietal90}; the central bands merge for $q_0$ even), so these polynomials of degree $q$ are equal: $\wt\Delta(E) = \Delta^2(E)$ for each $E$.
\end{proof}

\subsection{Measure of the spectrum}
The rest of the proof follows the argument of \cite{Beckeretal19}, namely it
uses  Lidskii's inequalities to bound
$|\widetilde{S}(\frac{p}{q})|$. The key observation is that choosing $\ell$ appropriately, we can make the corner elements of the matrix $B_{\th,k,\ell}$ very small, of order $1/q$ when $q$ is large. This is not possible to do in the standard representation for the almost Mathieu operator.
 Here are the details.

\bigskip

\underline{Case I: q is odd.}
Assume without loss that $(-1)^{(q+1)/2}>0$,  $\th\in (0,1/(4q)]$. (If $(-1)^{(q+1)/2}<0$, the analysis is similar.)
Then the eigenvalues  $\{ \lambda_i(\theta) \}_{i=1}^q$ of $B_{\th,k=0,\ell}$ labelled in decreasing order
are the edges of the spectral bands where $\wt\Delta(E)$ reaches its maximum $4 \sin(2 \pi q \theta)$  on the band;
and  the eigenvalues  $\{ \wh\lambda_i(\theta) \}_{i=1}^q$ of $B_{\th,k=\pi/q,\ell}$ labelled in decreasing order
are the edges of the spectral bands where $\wt\Delta(E)$ reaches its minimum $-4 \sin(2 \pi q \theta)$ on the band.
Then
\begin{equation} \label{eq4}
\begin{split}
|\sigma(\widetilde{H}_{\frac{p}{q}, \theta})| &= \sum_{j=1}^q (-1)^{q-j} (\wh\lambda_j(\theta) - \lambda_j(\theta))
=\sum_{j=1}^{(q+1)/2} (\wh\lambda_{2j-1}(\theta) - \lambda_{2j-1}(\theta))+
\sum_{j=1}^{(q-1)/2} (\lambda_{2j}(\theta) - \wh\lambda_{2j}(\theta))
;\\
&\wh\lambda_j(\theta) - \lambda_j(\theta)>0,\quad\mbox{if $j$ is odd} ;\qquad 
 \wh\lambda_j(\theta) - \lambda_j(\theta)<0,\quad\mbox{if $j$ is even.}
\end{split}
\end{equation}

Now we view $B_{\th,k=\pi/q,\ell}$ as $B_{\th,k=0,\ell}$ with the added perturbation
$$
B_{\th,k=\pi/q,\ell}-B_{\th,k=0,\ell}=
\begin{pmatrix}
& & -2b_{\ell+q-1} \\
& & \\
-2b_{\ell+q-1} & &
\end{pmatrix},
$$
which has the eigenvalues $\{ E_i(\theta) \}_{i=1}^q$ given by:
$$E_q(\theta) = -2|b_{\ell+q-1}(\theta)| < 0 = E_{q-1}(\theta) = \cdots = E_{2}(\theta) = 0 < 2|b_{\ell+q-1}(\theta)| = E_1(\theta).$$

The Lidskii inequalities (e.g., (2.51) in \cite{Beckeretal19}) are:
\begin{theorem} \label{thm1}
For any $q \times q$ self-adjoint matrix $M$, we denote is eigenvalues by $E_1(M) \geq E_2(M) \geq \cdots \geq E_q(M)$. For $q \times q$ self-adjoint matrices $A$ and $B$, we have:
$$
\begin{array}{l}
E_{i_1}(A+B) + \cdots + E_{i_m}(A+B) \leq E_{i_1}(A) + \cdots + E_{i_m}(A) + E_1(B) + \cdots + E_m(B); \\
E_{i_1}(A+B) + \cdots + E_{i_m}(A+B) \geq E_{i_1}(A) + \cdots + E_{i_m}(A) + E_{q-m+1}(B) + \cdots + E_q(B) ,
\end{array} 
$$
for any $1 \leq i_1 < \cdots < i_m \leq q$. 
\end{theorem}

Applying these inequalities with $A=B_{\th,k=0,\ell}$, $B=B_{\th,k=\pi/q,\ell}-B_{\th,k=0,\ell}$ gives:
\[
\begin{aligned}
\sum_{j=1}^{(q+1)/2} (\wh\lambda_{2j-1}(\theta) - \lambda_{2j-1}(\theta))&\le \sum_{j=1}^{(q+1)/2} E_j(\theta)=E_1(\th);\\
\sum_{j=1}^{(q-1)/2} (\lambda_{2j}(\theta) - \wh\lambda_{2j}(\theta))&\le -\sum_{j=(q-1)/2}^q E_j(\theta)=-E_q(\th).
\end{aligned}
\]
Substituting these into (\ref{eq4}), we obtain:
\begin{equation} \label{eq6}
|\sigma(\widetilde{H}_{\frac{p}{q}, \theta})| \leq E_1(\theta) - E_q(\theta) = 4|b_{\ell+q-1}(\theta)|.
\end{equation}
Moreover, by the invariance of $D_{\theta, k}(E)$ under the mapping $b_n \mapsto b_{n+m}$, for $n=0, 1, \ldots, q-1$ and any $m$, we can choose any $\ell$ in (\ref{eq6}), so that
\begin{equation} \label{eq7}
|\sigma(\widetilde{H}_{\frac{p}{q}, \theta})| \leq 4\min_{\ell}|b_{\ell+q-1}(\theta)|.
\end{equation}

In particular,
\begin{equation} \label{eq8}
\left| \widetilde{S}\left(\frac{p}{q}\right) \right|=|\sigma(\widetilde{H}_{\frac{p}{q}, \theta = \frac{1}{4 q}})|
\leq 4\min_{\ell}\left|b_{\ell+q-1}\left(\frac{1}{4q}\right)\right|=4\cdot 2 \left| \sin2 \pi \left(\frac{1}{4 q}\right)\right| \leq \frac{4\pi}{q}.
\end{equation}
Therefore, $| S(\frac{p_0}{q_0}) | = | \wt S(\frac{p}{q}) | \leq \frac{4\pi}{q} = \frac{4\pi}{q_0}$, as required.

\vspace{0.3cm}

\underline{Case II: q is even.} This case is similar to case I, so we omit some details for brevity.
This time, the Lidskii equations of Theorem \ref{thm1} show that $|\widetilde{S}(\frac{p}{q})| \leq \frac{8 \pi}{q}$. Indeed, as in (\ref{eq7}), we have (note the doubling of the eigenvalues for $\wt\Delta(E)=0$)
\begin{equation} \label{eq12}
|\sigma(\widetilde{H}_{\frac{p}{q}, \theta})| \leq 4\min_{\ell}|b_{\ell+q-1}(\theta)|.
\end{equation}

In particular,
\begin{equation} \label{eq15}
\left| \widetilde{S}\left(\frac{p}{q}\right) \right|=|\sigma(\widetilde{H}_{\frac{p}{q}, \theta = \frac{1}{2 q}})|
\leq 4\min_{\ell}\left|b_{\ell+q-1}\left(\frac{1}{2q}\right)\right|=4\cdot 2 \left| \sin2 \pi \left(\frac{1}{2 q}\right)\right| \leq \frac{8\pi}{q}.
\end{equation}
Therefore, $| S(\frac{p_0}{q_0}) | = | \wt S(\frac{p}{q}) | \leq \frac{8\pi}{q} = \frac{4\pi}{q_0}$, as required.

This completes the proof of Theorem \ref{Thm1}.

\vspace{0.3cm}

\section*{Acknowledgment}
   The work of S.J. was
 supported by NSF DMS-1901462. The work of I.K. was supported by  the Leverhulme Trust
research programme grant RPG-2018-260.

\end{document}